\documentclass[absolute,12pt]{mymathart}
\usepackage{mymathmacros,graphicx,color,commath}
\usepackage[shortlabels]{enumitem}
\newcommand{\commentout}[1]{}

\numberwithin{equation}{section}

 \newtheoremstyle{numberedstyle}
   {9pt}
   {9pt}
   {\normalfont}
   {}
   {\bfseries}
   {.}
   {\newline}
   {}

\newcommand{\HH}{\mathbb{H}}

\newcommand{\V}{\mathcal{V}}

\newcommand{\modd}[1]{|{\rm d}#1|}

\newcommand{\area}{\operatorname{area}}
\newcommand{\const}{\operatorname{const}}

\newcommand{\areacyl}{\area_{\operatorname{cyl}}}

\theoremstyle{changebreak}

\newtheorem{thm}{Theorem}[section]%
\newtheorem{lem}[thm]{Lemma}%
\newtheorem{cor}[thm]{Corollary}%
\newtheorem{prop}[thm]{Proposition}%
\newtheorem{conj}[thm]{Conjecture}%
\newtheorem{obs}[thm]{Observation}%

\newtheorem{proposition}[thm]{Proposition}

\newcommand{\edim}{\operatorname{edim}}

\theoremstyle{defnbreak}
\newtheorem{question}[thm]{Question}%
\newtheorem{defn}[thm]{Definition}%

\newcommand{\B}{\mathcal{B}}
\newcommand{\classS}{\mathcal{S}}

\usepackage{hyperref}

\newcommand{\id}{\operatorname{id}}

\title[Invariance of order and the area property]{On invariance of order  and the area property \\ for finite-type entire functions}

\author{Adam Epstein}
\address{Mathematics Institute, University of Warwick, Coventry CV4 7AL, UK}
\email{a.l.epstein@warwick.ac.uk}
\author{Lasse Rempe-Gillen}
\address{Dept.\ of Mathematical Sciences, University of Liverpool, Liverpool L69 7ZL, UK}
\email{l.rempe@liverpool.ac.uk}
\thanks{The second author was supported by EPSRC Fellowship EP/E052851/1 and a Philip Leverhulme Prize.}

\renewcommand{\P}{\mathcal{P}}
\newcommand{\hyp}{\operatorname{hyp}}

\begin{document}

\begin{abstract}Let $f\colon\C\to\C$ be an entire function that has only finitely many critical and asymptotic values. Up to \emph{topological equivalence}, the
  function $f$ is determined by combinatorial information, more precisely by an infinite graph known as a \emph{line-complex}. In this note,
  we discuss the natural question whether the order of growth of an entire function is determined by this combinatorial information.

 The search for conditions that imply a positive answer to this question leads us to the 
   \emph{area property}, which turns out to be related to many interesting and important questions
    in conformal dynamics and function theory. These include a conjecture of Eremenko and Lyubich, the measurable dynamics of entire functions, and
    pushforwards of quadratic differentials.

  We also discuss evidence that invariance of
   order and the area property fail in general.
\end{abstract}
\maketitle

 \section{Introduction} 

The \emph{order} $\rho(f)$ of a meromorphic function $f\colon\C\to \Ch$ is an
  important quantity in classical value-distribution theory
  \cite{nevanlinna}.
  In the special case where $f\colon\C\to\C$ is an entire function, 
   the order can be defined as 
\[
  \rho(f) \defeq \limsup_{z\to\infty}\frac{\log_+\log_+ |f(z)|}{\log |z|}
    \in [0,\infty]  
\]
 (where $\log_+ x = \max(0,\log x)$). 
Any polynomial or rational function has order $0$,
but there are also many transcendental entire and meromorphic functions with this
property. On the other hand, the maximum modulus of an entire function 
 can grow arbitrarily quickly; in particular, there are many functions
 of \emph{infinite order}. 

The set $S(f)$ of \emph{singular values} of an entire function $f$ is the 
  smallest closed set $S\subset\C$ such that
   \[ f\colon f^{-1}(\C\setminus S)\to \C\setminus S\]
 is a covering map; equivalently, $S(f)$ is the closure of the set of all critical and asymptotic values of $f$. 
 This set is of vital importance for both the function-theoretical 
  and iterative study of transcendental entire (and meromorphic)
  functions.

It is a guiding principle
 of both 1-dimensional holomorphic dynamics and 3-dimen\-sional hyperbolic
 geometry that combinatorics determines geometry, under suitable
 finiteness assumptions. In this note, we consider a potential extension of this principle to value-distribution 
 theory that was first proposed by the first author over fifteen years ago, but has not so far been discussed in print. 
 The natural setting for our considerations is the \emph{Speiser class} 
    \[ \classS \defeq \{f\colon \C\to\C \text{ transcendental, entire}\colon S(f)\text{ is finite}\}, \]
 which has been extensively studied both in dynamics and function theory. We shall also refer to such functions 
  as \emph{finite type maps}, in adopting terminology standard in holomorphic dynamics. We caution that the word 
\emph{type} has an entirely different meaning in value-distribution theory. 
   It is well-known that every function $f\in\classS$ has $\rho(f)\geq 1/2$; 
    compare~\cite[Proof of Corollary~2]{walteralexsingularities} or 
    \cite[Lemma~3.5]{ripponstallardmeromorphicdimensions}.
 
To a function $f\in\classS$, one can associate an infinite planar graph, known as the \emph{line-complex}, which encodes the topological mapping
  behaviour of $f$.\footnote{
{
We remark that the line-complex is uniquely defined only if one additionally fixes a suitable marking, which can be represented 
 by a Jordan curve through the singular values.}}
{%
   From this combinatorial data, 
  one can reconstruct the function $f$, up to pre-\ and post-composition by
   homeomorphisms. That is, two functions $f,g\in\classS$ 
   have the same line-complex if and only if they are \emph{topologically equivalent}
   in the sense of Eremenko and Lyubich:}

 \begin{defn}[Topological equivalence] \label{defn:topequ}
  Two entire functions $f$ and $g$ are called \emph{topologically equivalent} if there are order-preserving homeomorphisms $\phi$ and $\psi$ such that
    $\psi\circ f = g\circ \phi$.
 \end{defn}

{
  For this reason, we shall not require the formal definition of line-complexes~--
   for which we refer the reader to \cite[Chapter~7]{goldbergostrovskii}~--
   but will instead use the notion of topological equivalence, 
    which is easier both to define
    and to work with in our context. (See Section~\ref{sec:examples} for a discussion of 
    the properties of topological equivalence.)}

 It is natural to ask which properties of $f$ are combinatorially determined, and, in particular, whether this is the case for the order:

 \begin{question}[Invariance of order] \label{question:order}
   Let $f\in\classS$ with $\rho(f)<\infty$, and let $g$ be topologically equivalent to $f$. Is $\rho(f)=\rho(g)$?
 \end{question}

\smallskip

 For transcendental \emph{meromorphic} functions, the 
order is given by 
\[
\rho (f)=\limsup_{r\rightarrow \infty }\frac{\log T(r,f)}{\log r} 
\]
where $T(r,f)$ is the \emph{Nevanlinna characteristic} of $f$. (We 
  emphasize that knowledge of Nevanlinna theory will not be required for the
  remainder of the paper.) 
  Question \ref{question:order} is partly motivated by the fact
  that the answer is positive in an important, classical
  case: that of meromorphic functions with
   rational Schwarzian derivative. {Indeed, 
   in this case the order can be directly determined from the combinatorial information
   of the function~-- more precisely, $\rho(f)=\ell/2$, where $\ell$ is the number of \emph{logarithmic singularities}~-- and hence is indeed invariant. (See 
   Corollary \ref{cor:rationalschwarzianinvariant} and the discussion that precedes it.)}
  {More generally, there are a number of classical subclasses of $\classS$ that were defined in terms
    of the structure of their line-complexes  (e.g.\ maps with finitely many \emph{simply- and doubly periodic ends} \cite{wittich}, 
    and more generally \emph{asymptotically periodic ends}
   \cite{goldbergostrovskii}). In these cases, it seems to have been implicitly understood that the order depends only on the 
     line complex, and hence that invariance of order holds for these classes.
  However,  the answer to Question \ref{question:order} for
  \emph{general} meromorphic finite-type functions is  negative 
   by work of K\"unzi \cite{kunziorder}.}

In addition to the Speiser class, the larger \emph{Eremenko-Lyubich class}
   \[ \B \defeq \{f\colon \C\to\C \text{ transcendental, entire}\colon S(f)\text{ is bounded}\} \]
 has also been studied extensively in complex dynamics. 
 In this class, there do exist cases where the order is not
 invariant under topological equivalence. 

\begin{thm}[Counterexamples in class $\B$] \label{thm:changeinB}
 There exist two finite-order functions {$f,g\in\B$} such that
  {$f$ and $g$} are topologically equivalent, but {$\rho(f)\neq \rho(g)$}. 
\end{thm}

{These examples arise from complex dynamics; 
   more precisely they are given by
  \emph{Poin\-car\'e functions} (maximally extended linearizing maps) associated to the repelling 
   periodic cycles of polynomials. 
   As we shall see, for a polynomial with connected Julia set, these Poincar\'e functions belong to 
  the Eremenko-Lyubich class. Furthermore,
   an orientation-preserving topological conjugacy between polynomials induces topological 
   equivalences between corresponding Poincar\'e functions
    (Proposition \ref{prop:linearizerstop}).
   On the other hand, in this situation the order is determined by the multiplier of the associated 
   periodic cycle, which may change under a topological
   conjugacy; see Corollary~\ref{cor:changeinB}.}

 Note that this construction cannot be extended to yield counterexamples
  to invariance of order in the class $\classS$. Indeed, a Poincar\'e function
  is in $\classS$ if and only if the corresponding polynomial
  is \emph{postcritically finite}, but postcritically finite maps
  are rigid by the Thurston Rigidity Theorem  \cite{douadyhubbardthurston}. We are able to 
  give a purely function-theoretic explanation of this phenomenon by considering
  an important geometric property for entire functions.

\subsection*{The area property}
 The following result, which is a consequence
  of the
   well-known Teichm\"uller-Wittich Theorem, will allow us to 
  verify invariance of order for certain functions $f\in\classS$. 
 \begin{thm}[Invariance of order and the area property] \label{thm:area}
   Let $f\in\classS$, and suppose that 
    \begin{equation}\label{eqn:area}
       \int_{f^{-1}(K)\setminus \D} \frac{\dif x \dif y}{|z|^2} < \infty 
    \end{equation}
  for every compact set $K\subset \C\setminus S(f)$. 

  Then the order of $f$ is invariant under topological equivalence.
   (Here $\D=\{|z|<1\}$ denotes the unit disc.)
 \end{thm}

 The condition (\ref{eqn:area}) means that the \emph{cylindrical area}
    $\areacyl(f^{-1}(K)\setminus \D)$~-- 
  i.e.\ area with respect to the conformal metric $\dif s=\modd{z}/|z|$ on the
   punctured plane $\C^*$~-- is finite.
 Note that this condition 
  makes perfect sense not just for a function $f\in\classS$, but also for general 
  entire functions $f$, and in particular for 
    $f\in\B$:
 \begin{defn}[The area property]\label{defn:areaproperty}
   We say that an entire function $f$ has the \emph{area property} 
     if~\eqref{eqn:area} holds 
  for every compact set $K\subset \C\setminus S(f)$. 

  If $f\in\B$ and this property holds for every compact subset of the unbounded connected component of $\C\setminus S(f)$, we say that
   $f$ has the area property \emph{near infinity}.
 \end{defn}

 The area property and some variants thereof appear to be closely connected to
   a number of interesting questions in complex function theory and complex dynamics. In particular, a similar question was stated by 
   Eremenko and Lyubich:

 \begin{conj}[Detection of asymptotic values{\cite[p.\ 1009]{alexmisha}}] \label{conj:eremenkolyubich}
   Suppose that $f\in\classS$ is such that, for some $R>0$, 
       \begin{equation}\label{eqn:eremenkolyubich}
    \liminf_{r\to\infty} \frac{1}{\log r} \int_{\{z\in\C\colon 1\leq |z|\leq r\text{ and }|f(z)|\leq R\}} \frac{\dif x \dif y}{|z|^2} > 0.  \end{equation} 
   Then $f$ has a finite asymptotic value.
 \end{conj}
 In other words, suppose that $f\in\classS$ has no finite asymptotic values. Then Conjecture~\ref{conj:eremenkolyubich} 
   {would imply that the part of the 
   logarithmic area of $f^{-1}(K)\setminus \D$ at modulus at most $r$ does not grow too quickly (although the total logarithmic area is allowed
    to be infinite, in contrast to the area property)}. Conversely, suppose that the area property holds for $f$, and that 
   additionally the multiplicity of the critical
   points of $f$ is uniformly bounded. Then the integral in (\ref{eqn:eremenkolyubich}) is bounded,
   and hence $f$ satisfies Conjecture \ref{conj:eremenkolyubich} (see Lemma~\ref{lem:boundedcriticality}).

 The area property is often easy to verify, allowing us to establish a positive answer to Question \ref{question:order} in such cases. 
  In particular, we can show that the Poincar\'e function $f$ for a polynomial $p$ 
  with connected Julia set typically
  has the area property, even when invariance of order fails. Recall that, as mentioned above, such $f$ belongs to the class $\B$
   and has finite order (see~\eqref{eqn:logD} below). 
 \begin{thm}[The area property for linearizers] \label{thm:linearizers}
   Let $f\in\B$ be the Poincar\'e function associated to a repelling periodic 
   point of a polynomial $p$ with connected Julia set.
   Then $f$ has the area property near infinity. Furthermore, 
    $f$ has the area property if and only if 
    $p$ does not have any Siegel discs.

  In particular, if $p$ is postcritically finite, then $f\in\classS$ and the order of $f$ is invariant under topological equivalence. 
 \end{thm}

The preceding theorem provides examples of functions in the class $\B$ where 
  the area property fails. These examples 
  rely in an essential way on the fact that the singular set of $f$
  (which includes the boundary of any Siegel disc of $p$) disconnects 
  the plane. So the theorem leaves open the possibility that 
   the area property holds 
   \emph{near infinity} for all finite-order functions $f\in\B$, 
   which would imply the general area property, and hence invariance of order, when $f\in\classS$. 
   We show that this is not the case:
 \begin{thm}[Counterexamples to the area property near infinity] \label{thm:counterexample}
  There exists a function $f\in\B$ of finite order such that $f$ does not
  have the area property near infinity. 
 \end{thm}
 \begin{remark}
  The counterexample is constructed by precisely the 
   same method as that used by the second author to construct a 
   hyperbolic entire function with maximal
   hyperbolic dimension constructed in 
   \cite{approximationhypdim}. Indeed, it can be shown that the
    counterexample from that paper violates the area property. For some remarks
   concerning connections between the area property and measurable dynamics,
   see Section~\ref{sec:remarks}. 
 \end{remark}

 We consider Theorem 
  \ref{thm:counterexample} to provide  strong evidence that
  invariance of order does not hold in general. 

\begin{conj}[Non-invariance of order] \label{conj:counterexample}
  There exist entire functions $f,g\in\classS$ such that
   $f$ and $g$ are topologically equivalent, but
   $\rho(f)\neq \rho(g)$. 
\end{conj}

\subsection*{Subsequent work}
 While this article was being prepared, a proof of Conjecture
  \ref{conj:counterexample} was announced by Chris Bishop \cite{bishoporder}. His work 
  also includes a counterexample to Conjecture~\ref{conj:eremenkolyubich}. {%
   Since his results were announced, we noticed that Poincar\'e functions can be used to give 
   an alternative counterexample to the latter conjecture; we include the short argument in
   Proposition~\ref{thm:eremenkolyubichcounterexample}.}

 The area property is also closely connected with the question of whether the exceptional set in a certain 
   theorem
   of Littlewood can be chosen to be finite. Geyer \cite{geyerlittlewood} has independently considered Poincar\'e functions for
   polynomials with Siegel discs, as in Theorem \ref{thm:linearizers}, to construct class $\B$ counterexamples to this property. 
   We refer to his paper for a discussion of the precise question. 

\subsection*{Structure of the article}
  In Section \ref{sec:examples}, we introduce a number of definitions and preliminaries, and in particular recall some 
   key facts concerning topological and quasiconformal equivalence. We also review the classical case of maps with polynomial or
    rational
   Schwarzian derivative, as studied by Nevanlinna and Elfving. In Section~\ref{sec:linearizers}, we study the 
   basic properties of Poincar\'e functions and prove Theorem~\ref{thm:changeinB}. The area property is studied
   in detail in Section \ref{sec:area}, where we prove Theorem~\ref{thm:area} and discuss a number of equivalent 
   formulations of~\eqref{eqn:area} that will be helpful in the following. We then return to the study of Poincar\'e functions and the
   proof of Theorem \ref{thm:linearizers}, which turns out to rely on a connection between the area property for the linearizer and the 
   Poincar\'e series of the original polynomial. Theorem \ref{thm:counterexample} is proved in Section \ref{sec:counterexample}, using 
   a construction from \cite{approximationhypdim}.

  Finally, Section \ref{sec:qd} discusses consequences of the area property for pushforwards of \emph{quadratic differentials}, and
   Section \ref{sec:remarks} touches on a number of topics that are connected to our considerations, but go beyond the main scope of the
   article.

\subsection*{Basic notation}
  We shall assume that the reader is familiar with basic facts from
   complex geometry \cite{forster}, hyperbolic geometry
   \cite{beardonminda} and the theory
   of quasiconformal maps \cite{ahlforsqc}. In particular, we shall use the following elementary fact.
\begin{obs}[Quasiconformal maps isotopic to a given homeomorphism]\label{obs:isotopy}
 Let $\phi:\Ch\to\Ch$ 
   be an orientation-preserving homeomorphism, and let $E\subset\Ch$ be finite.

 Then
  there exists a quasiconformal homeomorphism $\tilde{\phi}:\Ch\to\Ch$ isotopic to $\phi$ relative $E$ and conformal near $E$. 
   If $\# E \leq 3$, then $\tilde{\phi}$ can be chosen to be a M\"obius transformation. 
\end{obs}
 If $\phi$ is a quasiconformal map, then (following Bishop) the \emph{quasiconstant} of $\phi$
    is the smallest number $K$ such that $\phi$ is $K$-quasiconformal. Furthermore,
    we denote the complex dilatation of $\phi$ by $\mu_{\phi}$. 

The Koebe Distortion Theorem 
   \cite[Theorem~1.3]{pommerenke} will also be used frequently. 
   An important consequence of this theorem (and the Schwarz Lemma) is 
   the \emph{standard estimate} \cite[Theorems 8.2 and 8.6]{beardonminda}
   on the hyperbolic metric in a simply-connected domain: if $V\subset\C$ is simply-connected and $\rho_V$ denotes the density of the hyperbolic
   metric of $V$, then 
       \begin{equation}\label{eqn:standardestimate} \frac{1}{2\dist(z,\partial V)} \leq \rho_V(z) \leq \frac{2}{\dist(z,\partial V)}. \end{equation}

  Throughout the paper, the complex plane, the punctured plane,
   the Riemann sphere and
   the unit disc are denoted $\C$, $\C^*$, $\Ch$ and $\D$, respectively. 

  We shall sometimes use ``$\const$'' to indicate a constant in a formula. For example, $f(x) \leq \const\cdot |x|$ should be read as 
   ``there exists a constant $C$ such that $f(x)\leq C\cdot|x|$''. We also write $a\asymp b$ to mean that $a$ and $b$ differ by at most a 
   multiplicative constant; i.e. $a \leq \const\cdot\, b \leq \const\cdot\, a$. The notation $f(z)\sim g(z)$ (as $z\to\infty$) means that
   $\lim_{z\to\infty} f(z)/g(z) = 1$.

\subsection*{Acknowledgments} 
We thank {Chris Bishop}, David Drasin, {Lukas Geyer}, Kevin
Pilgrim, Stas Smirnov and, especially, Alex Eremenko {interesting discussions, encouragement
and assistance. We would also like to thank the referees for helpful comments and corrections.}

\section{Maps in the Speiser and Eremenko-Lyubich class}

\label{sec:examples}

\subsection*{Singular values}
 Let $f\colon\C\to\C$ be a transcendental entire function. A point
  $z\in\C$ is called a \emph{regular value} if 
  there is an open {$U\ni z$} such that 
  $f$ maps each component of $%
  f^{-1}(U)$ homeomorphically onto $U$. Otherwise $z$ is called a
  (finite) \emph{singular value}, and the set of all such singular values
  is denoted by $S(f)$. Note that, since the set of regular values is open,
  this coincides with the definition of $S(f)$ given in the introduction.

 Denote the sets of \emph{critical} and \emph{asymptotic} values
  of $f$ by 
\begin{align*}
C(f)&\defeq\left\{ x\in \C\colon x=f(w)\text{ for some }w\in \C\text{ with }f^{\prime
}(w)=0\right\}\qquad\text{and}\\
A(f)&\defeq\left\{ 
\begin{array}{c}
x\in \C\colon x=\lim_{t\rightarrow 1}f(\gamma (t))\text{ for some path } \\ 
 \gamma \colon [0,1)\rightarrow \C\text{ with $\gamma(t)\to\infty$ as $t\to 1$}
\end{array}
\right\},
\end{align*}
 respectively. Then it follows from covering theory that
   \[ S(f) = \overline{C(f)\cup A(f)}. \]
  (Clearly $C(f)\cup A(f)\subset S(f)$, and if $x$ has a neighborhood 
   not intersecting $C(f)\cup A(f)$, then $x$ is a regular value by
   the Monodromy Theorem.) 

\subsection*{Topological and quasiconformal equivalence}
  Note that all three sets, $C(f)$, $A(f)$ and $S(f)$, are defined
 topologically, and hence are preserved by topological equivalence. 

\begin{obs}[Topological equivalence respects singular values]
 Suppose that $f$ and $g$ are topologically equivalent, say
  $\psi\circ f = g\circ\phi$. Then
  $A(g)=\psi(A(f))$, $C(g)=\psi(C(f))$ and $S(g)=\psi(S(f))$.
\end{obs}

Recall the definition of the \emph{Speiser class} and the
  \emph{Eremenko-Lyubich class} from the introduction:  
\begin{align*}
   \classS &\defeq \{f\colon\C\to\C \text{ transcendental, entire}\colon S(f)\text{ is finite}\};\\
    \B &\defeq \{f\colon\C\to\C \text{ transcendental, entire}\colon S(f)\text{ is bounded}\}. 
\end{align*}

 One of the key properties of the class $\classS$ with respect to
  topological equivalence is that the 
  maps $\phi$ and $\psi$ in Definition \ref{defn:topequ} can always be
  chosen to be quasiconformal (see Proposition \ref{prop:topequ}~\ref{item:qcequ} below). 
  For functions with infinitely many singular values, this need no longer be 
  true, and it makes sense to introduce the following
  definition (see \cite{boettcher}):
 \begin{defn}[Quasiconformal equivalence]
   Two entire functions $f$ and $g$ are called \emph{quasiconformally equivalent} if there are quasiconformal homeomorphisms
     $\phi,\psi\colon \C\to\C$ such that $\psi\circ f = g\circ \phi$.

   We shall refer to $\phi$ and $\psi$ from this definition or from 
    Definition \ref{defn:topequ} as \emph{witnessing homeomorphisms}. 
 \end{defn}

 {The following facts regarding topological and quasiconformal equivalence are mostly folklore
   (although, for the class $\classS$,  parts~\ref{item:isotopy} and~\ref{item:qcequ} essentially
    appear in \cite[Section~3]{alexmisha}). 
     We provide the short proofs for completeness.}
 \begin{prop}[Properties of topological equivalence]
  \begin{enumerate}[{\normalfont(a)}]
    \item Suppose that $f$ and $g$ are topologically equivalent, with witnessing homeomorphisms $\phi$ and $\psi$. If a homeomorphism
      $\tilde{\psi}\colon\C\to\C$ is
       isotopic to $\psi$ relative $S(f)$, then there exists a homeomorphism $\tilde{\phi}$, isotopic to $\phi$ relative
      $f^{-1}(S(f))$, such that
      $\tilde{\phi}$ and $\tilde{\psi}$ are also witnessing homeomorphisms for $f$ and $g$. If $\tilde{\psi}$ is quasiconformal, respectively conformal,
     then $\tilde{\phi}$ is also. \label{item:isotopy}
    \item If $f$ and $g$ are quasiconformally equivalent and $f$ has finite positive order, then $g$ also has finite positive order. More precisely,
  \[
     0<\frac{1}{K}\leq \frac{\rho (g)}{\rho (f)}\leq K<\infty,
\]
   where $K$ is the quasiconstant of $\phi$. \label{item:order}
    \item Suppose that $f$ and $g$ are quasiconformally equivalent and that the 
      witnessing homeomorphism $\phi$ is \emph{Lipschitz at $\infty$}; i.e.,
      $|\phi(z)|\asymp |z|$ for sufficiently large $z$.
      Then
      $\rho(f)=\rho(g)$. \label{item:asconforder}
    \item If $f,g\in\classS$ are topologically equivalent, then they are quasiconformally equivalent. If $\# S(f)=2$, then $\phi$ and $\psi$ can be chosen
     to be affine (and $\rho(f)=\rho(g)$). \label{item:qcequ}
  \end{enumerate}\label{prop:topequ}
 \end{prop}
 \begin{remark}
   It follows from the final statement in \ref{item:qcequ} that the answer to Question~\ref{question:order} is always positive when
     $\# S(f)=2$. 
 \end{remark}
 \begin{proof}
  Part~\ref{item:isotopy} follows by lifting the isotopy. More precisely, let $(\psi_t)_{t\in[0,1]}$ 
   be an isotopy from $\psi$ to $\tilde{\psi}$. Then, on every component
     of $U \defeq f^{-1}(\C\setminus S(f))$, there is a unique lift $(\phi_t)_{t\in[0,1]}$ 
     of this isotopy (since $f$ is a covering map on each such component). So we
     have an isotopy $\phi_t\colon U\to \phi(U)$, and only need to show that the maps $\phi_t$ extend continuously to $\partial U$ and agree with $\phi$ there.

   We may assume without loss of generality that $\psi=\phi=\id$. 
     Let $z_0\in \partial U$, and set $w_0 \defeq f(z_0)$; we must show that
    $\phi_t(z)\to z_0$ as $z\to z_0$ in $U$. Let $D$ be a small simply-connected neigborhood of $z_0$, chosen to 
     ensure that $f\colon D\to f(D)\eqdef V$ is a proper map with no critical points except possibly at $z_0$. We must show that $\phi_t(z)\in D$ when
     $z$ is sufficiently close to $z_0$. By continuity of $f$ and the isotopy, if $z$ is close enough to $z_0$, and $w\defeq f(z)$, then $\psi_t(w)\in V$ for all $t$.
      (Recall that $\psi_t(w_0)=w_0$ for all $t$.) The point $z_t=\phi_t(z)$ is obtained by analytic continuation of $f^{-1}$ along the 
      curve $t\mapsto \psi_t(w)$, which is entirely contained in $V$. Hence it follows that $z_t\in D$ for all $t$, as desired.

   Away from the critical points of $f$,
     the homeomorphism $\phi$ can be written as a composition of $\psi$ with an inverse branch of $f^{-1}$. Hence, 
    if $\psi$ is quasiconformal resp.\ conformal, then $\phi$ is also (with the same quasiconstant).

  Claim \ref{item:order} follows from the H\"older property of quasiconformal
     mappings (see e.g.\ \cite[Theorem 2 in Chapter III]{ahlforsqc}). Indeed, we have $|w|^{1/K}/C\leq |\phi(w)| \leq C\cdot |w|^K$ for a suitable constant
     $C$ and all sufficiently large $w$, and similarly for $\psi$. To use this in the formula for the
     order of $g$, let us write $z=\phi(w)$. We have 
      \begin{align*} \frac{\log_+\log_+ |g(z)|}{\log|z|} &=
            \frac{\log_+\log_+ |\psi(f(w))|}{\log |\phi(w)|} \leq
            \frac{\log_+\log_+ C\cdot |f(w)|^K}{\log \frac{|w|^{1/K}}{C}} \\ &=
            K\cdot \frac{O(1) + \log_+\log_+ |f(w)|}{O(1) + \log |w|}, \end{align*}
       and hence $\rho(g)\leq K\rho(f)$. The opposite inequality follows on reversing 
     the roles of $f$ and $g$. 
    Item \ref{item:asconforder} is immediate from the same computation. 

   The final claim follows from \ref{item:isotopy} and Observation \ref{obs:isotopy}. 
 \end{proof}%

\subsection*{Maps with polynomial Schwarzian derivative}
 The investigations of F.\ and R.\ Nevanlinna concerning the inverse problem of 
  value-distribution
  theory 
  involved a study \cite{fnevanlinnaschwarzian,nevanlinnaschwarzian}
  of those transcendental meromorphic functions $f\colon\mathbb{C%
}\rightarrow \widehat{\mathbb{C}}$ whose \emph{Schwar\-zian derivative }$\mathcal{%
S}_{f}=\left( \frac{f^{\prime \prime }}{f^{\prime }}\right) ^{\prime }-\frac{%
1}{2}\left( \frac{f^{\prime \prime }}{f^{\prime }}\right) ^{2}$ is a
polynomial. They characterized these maps as the meromorphic functions with
finitely many ``logarithmic ends'', or logarithmic singularities: that is, $f$ is a map of finite
 type, all singular values of $f$ are asymptotic (rather than critical) values, and furthermore
 for every $a\in S(f)$ and every sufficiently small disc $D\ni a$, 
 the number of connected components
 of $f^{-1}(D)$ that are not mapped homeomorphically by $f$ is finite. Such a component
 is called a ``logarithmic tract'' and corresponds to a unique ``logarithmic singularity'' (see \cite{walteralexsingularities} for more 
  details concerning the classification of inverse function singularities). 

For entire functions, the
condition on $\mathcal{S}_{f}$ reduces to the requirement that $f$ has
polynomial \emph{nonlinearity }$\mathcal{N}_{f}=\frac{f^{\prime \prime }}{%
f^{\prime }}$. 

Slightly more generally, G.\ Elfving \cite{elfving}
allowed finitely many
critical points in addition to the finitely many logarithmic singularities, to obtain
the class of transcendental meromorphic functions $f\colon\Ch \to \Ch$
with rational Schwarzian derivative. (Compare also \cite{adamschwarzian,langleyschwarzian}.) 
The finite poles of $\mathcal{S}_{f}$
are precisely the critical points of $f$; in fact,
  \[ \mathcal{S}_{f}(\zeta )=%
    \frac{m}{(z-\zeta )^{2}}+O\left( \frac{1}{z-\zeta }\right)\]
near a point $\zeta$ where $\deg _{\zeta }f=m$.
The corresponding entire functions have
rational nonlinearity, with 
  \[ \mathcal{N}_{f}(\zeta )=\frac{m}{z-\zeta }%
+O\left( 1\right)\] near such a point $\zeta $. 

{A calculation of asymptotics
 from the defining differential equations (see pp.
 298-303 of \cite{nevanlinna}, and pp. 391-393 of \cite{hille}) allows one to
 determine the order of these functions explicitly in terms of the degree of the Schwarzian at $\infty$. The latter,
  in turn, can be expressed in terms of the number of logarithmic singularities of $f$:}

\begin{proposition}[Order of maps with rational Schwarzian]\label{prop:schwarzian}
Let $f\colon \C \to \Ch $ be a transcendental meromorphic function with
rational Schwarzian derivative. {Then 
  $\rho (f)=\frac{2+\deg_{\infty }\mathcal{S}_{f}}{2} = \frac{\ell}{2}$, where $\ell$ is the number of logarithmic singularities of $f$.}
In particular,
if $f$ is entire with rational nonlinearity, then 
 $\rho (f)=1 + {\deg_{\infty} \mathcal{N}_{f}}$. 
\end{proposition}

\begin{cor}[Invariance of order for maps with rational Schwarzian]
   \label{cor:rationalschwarzianinvariant}
 Let $f\colon \C \to \Ch $ be a transcendental meromorphic function with
rational Schwarzian derivative. 
Then $\rho (f)=\rho (g)$ for any topologically equivalent map 
 $g$.
\end{cor}
\begin{proof}
 Clearly the numbers of logarithmic singularities over infinity, logarithmic singularities over finite asymptotic values, and of
  critical points are preserved
 under topological equivalence. In particular, the function $g$ also has rational Schwarzian derivative, and hence 
 $\rho(f)=\rho(g)$ by Proposition \ref{prop:schwarzian}
\end{proof} 

\section{Poincar\'e functions: Non-invariance of order in \texorpdfstring{$\B$}{B}}\label{sec:linearizers}

 Let $h\colon \C\to\C$ be an entire function, and let $\zeta\in\C$ be a repelling
  fixed point
  of $h$. That is, $h(\zeta)=\zeta$ and $|\lambda|>1$,
    where $\lambda =h^{\prime }(\zeta )$ is the associated \emph{multiplier}.

 Then 
  there exists a unique (up to restriction) 
  conformal map $f$, defined near $0$, such that
  $f(0)=\zeta$, $f'(0)=1$ and
   \begin{equation} \label{eqn:functionalequation}
      f(\lambda z) = h(f(z)). \end{equation}
  (See e.g.\ \cite[Theorem 6.3.2]{beardon}.)
 Using (\ref{eqn:functionalequation}), we can extend
  $f$ to an entire function $\C\to\C$ satisfying
  (\ref{eqn:functionalequation}) for all $z\in\C$.

\begin{defn}[Poincar\'e function]
 The linearizing semiconjugacy $f\colon \C\to\C$ as above is called
  the \emph{Poincar\'e function} of $f$ at $\zeta$. 
\end{defn}

If $h$ is a polynomial of degree $D$, then it is easy to verify that
  $f$ has finite order 
\begin{equation}
\rho (f)=\frac{\log D}{\log |\lambda |}  \label{eqn:logD}
\end{equation}
(see \cite[Formula~(4)]{eremenkosodin}). 
 Moreover, the singular set $S(f)$ coincides with the postcritical set $%
\P(h) = \cl{\bigcup_{n=1}^{\infty }h^{n}(C(h))}$ of $h$ 
  \cite[Proposition~3.2]{helenajoern}\footnote{%
{We remark that, in the proof of part (ii) of \cite[Proposition~3.2]{helenajoern}, an
   equality is stated for the singular set $S(h\circ f)$ (using our notation above) that does not
  appear justified in the case where $h$ is transcendental entire. However, this equality is not, in fact, used later in the
  proof, so that the proposition remains correct as stated. 
   (Also note that, both in our paper and in \cite{helenajoern}, the result is usually applied only 
  when $h$ is a polynomial.)}}
. In particular, $f$ has finite type if
  and only if $h$ is \emph{postcritically finite}, and $f$ belongs to the 
  Eremenko-Lyubich class if and only if 
   $\P(h)$ is bounded, which is equivalent to $J(h)$ being connected. Further 
   function-theoretic properties of Poincar\'e functions have been investigated by
   Drasin and Okuyama \cite{drasinokuyama}.

To prove Theorem \ref{thm:changeinB}, we observe that
 a conjugacy between polynomials (and, in fact, entire functions)
   will result in the topological equivalence of their linearizers.

\begin{prop}[Conjugacy implies equivalence of Poincar\'e functions] \label{prop:linearizerstop}
   Suppose that $h_1$ and $h_2$ are non-constant, non-linear entire functions, and that $h_1$ and $h_2$ are topologically conjugate
     via a homeomorphism
     $\psi\colon \C\to\C$; that is, $\psi\circ h_1 = h_2\circ \psi$. Let $x_1$ be a repelling  fixed
    point of $h_1$, set $x_2 \defeq \psi(x_1)$, and let
     $f_1,f_2\colon \C\to\C$ be the corresponding Poincar\'e functions of $h_1$ and $h_2$.

   Then there is a homeomorphism $\phi\colon \C\to\C$ such that $\psi\circ f_1 = f_2\circ \phi$. If $\psi$ is quasiconformal, then $\phi$ is also quasiconformal.
\end{prop}
\begin{proof}
Let 
   $\eta_2$ be the branch of $f_2^{-1}$ that takes $x_2$ to $0$.  We first 
    define $\phi(z)$, provided $z$ is not a critical point of $h_1$. 
    To do so, let $\alpha$ be a curve connecting $0$ and $z$ and not
    passing through any critical point of $f_1$. Then $\phi(z)$ is
     obtained by analytic continuation of $\eta_2$ along the curve
    $\psi\circ f_1\circ \alpha$.

  The fact that this analytic continuation is defined, and that it is independent of the curve $\alpha$, can be seen as follows. Let $\beta_1$
   be the concatenation of $\alpha$ with the reverse of $\hat{\alpha}$; then $\beta_1$
    is a closed curve
    beginning and ending at $0$. We must show that analytic continuation of $\eta$ along
    $\gamma \defeq \psi\circ f_1\circ \beta$ is possible and leads to $\eta$ rather than some
    other branch of $f_2^{-1}$. 
  
  In other words, 
   we must show that there is a curve $\beta_2$, beginning and ending at $0$ and not
   passing through any critical points of $f_2$, such that $f_2\circ\beta_2 = \gamma$.
    To do so, let $n$ be sufficiently large, and consider the curve
     \[ \gamma^n \defeq \psi \circ f_1\circ \lambda_1^{-n}\circ \beta_1, \]
    where $\lambda_1$ denotes (multiplication by) the multiplier of $h_1$ at $x_1$. 
    For sufficiently large $n$, the curve $\gamma^n$ is contained in a linearizing neighborhood
    of $h_2$ around $x_2$, so we can set $\beta_2^n \defeq \eta_2\circ \gamma^n$; this is a closed curve beginning and ending at $0$. Set
     \[ \beta_2 \defeq \lambda_2^n \circ \beta_2^n, \] 
    where $\lambda_2$ is the multiplier of $h_2$ at $x_2$. Then $f_2\circ \beta_2 = h_2^n\circ \gamma^n = \gamma$. Furthermore, since
    $\beta_1$ does not contain any critical points of $f_1$, and $\psi$ is a topological conjugacy (and hence sends critical points of $h_1$ to critical points
    of $h_2$), the curve $\beta_2^n$ does not contain any critical points of $h_2^n$. Hence $\beta_2$ does not contain any critical points of
    $f_2$, as claimed.

  This defines $\phi$ with the desired property on the complement of the set of critical points of $f_1$. It is easy to see (e.g.\ by applying the
    same construction, but reversing the roles of $f_1$ and $f_2$) that $\phi$ is a homeomorphism between
    the complement of the critical points of $f_1$ and the complement of the critical points of $f_2$. Since both sets are discrete,
   it follows that $\phi$ extends to a homeomorphism $\phi\colon \C\to\C$. (Alternatively, it is also easy to check directly that
   $\phi$ extends continuously to every critical point of $f_1$.)

  If $\psi$ is quasiconformal, then clearly $\phi$ is quasiconformal
  (as it is defined as a composition of locally quasiconformal maps). 
  In this case (which is the one we are mainly interested in), 
  there is an alternative and shorter proof of the proposition.
  Indeed, we can obtain the homeomorphism $\phi$ by solving
  the Beltrami equation for the pullback $f_1^* \mu_{\psi}$ of the complex dilatation 
  of $\psi$. Since $\mu_{\psi}$ is invariant under $h_1$, the
  pullback $f_1^* \mu_{\psi}$ is invariant under $\lambda_1$.  
  It follows that $\phi$ conjugates $\lambda_1$ to
  a linear map, and hence that $g \defeq \psi\circ f_1 \circ \phi^{-1}$ 
  semiconjugates $h_2$ to this linear map. Uniqueness of the
  Poincar\'e function implies $g=f_2$.
\end{proof}

Proposition \ref{prop:linearizerstop} implies Theorem \ref{thm:changeinB},
  in the following stronger form:
\begin{cor}[Non-invariance of order in class $\B$]\label{cor:changeinB}
  There exist two functions $f,g \in\B$ such that $f$ and $g$ are
   quasiconformally equivalent, but $\rho(f)\neq \rho(g)$.
\end{cor}
\begin{proof}
  Consider the family of quadratic polynomials
     \[ p_{a}\colon  z\mapsto a z + z^2, \]
    with $0<|2-a|<1$. Then $p_{a}$ has a repelling fixed point 
    of multiplier $a$ at $0$, and an attracting
    fixed point of multiplier $2-a$ at $1-a$. It is well-known
    that any two elements of this family are quasiconformally
    conjugate; see \cite[Proposition 23 on p.\ 135]{ahlforsqc}.

  Let $f_{a}$ be the Poincar\'e function for $p_{a}$ at $0$, and consider two 
    polynomials in this family whose multipliers have different moduli; for example 
    $f \defeq f_{3/2}$ and $g \defeq f_{4/3}$.  By
    Proposition \ref{prop:linearizerstop}, $f$ and $g$ are
    quasiconformally equivalent, but
    $\rho(f)\neq \rho(g)$ by \eqref{eqn:logD}.
\end{proof}

\section{The area property}\label{sec:area}

\subsection*{The area property implies invariance of order}
 The \emph{Teichm\"uller-Wittich} Theorem states that, 
  if $\phi\colon \C\to\C$  is quasiconformal and
\[
\lim_{R\rightarrow \infty }\int_{|z|>R}\left| \frac{\mu _{\phi }(z)}{z^{2}}%
\right| \dif x \dif y=0 
\]
then $|\phi(z)|\sim a\cdot |z|$ for some $a>0$. (Recall that  $\mu_{\phi}$ is the complex dilatation of $\phi$.) 

   Belinski and Lehto later
 also showed that $\arg \phi(z) - \arg z$ has a limit under the same assumptions,
 proving that $\phi$ is in fact \emph{asymptotically conformal} at infinity; i.e.\
    $\phi(z)\sim az$ for some $a\in\C$. This is the
    \emph{Teichm\"uller-Wittich-Belinski-Lehto Theorem}, sometimes
    also known as the Teichm\"uller-Belinski Theorem (see \cite[Chapter V, \S 6]{lehtovirtanen}).

 The Teichm\"uller-Wittich Theorem 
  almost immediately leads to the proof of Theorem \ref{thm:area}, which we 
  shall now state somewhat more generally. In particular, we note the fact that the
  area property itself is 
  preserved under suitable quasiconformal equivalence.

\begin{prop}[Area property, invariance of order and qc equivalence]
  Suppose that the entire functions 
    $f$ and $g$ are quasiconformally equivalent, with witnessing homeomorphisms $\phi$ and $\psi$
      such that the dilatation
    of $\psi$ is supported on a compact subset $K\subset \C\setminus S(f)$.
    If $f$ has the area property, then $\rho(f)=\rho(g)$, and $g$ also has the
    area property. 

   In particular, if $f$ belongs to the class $\classS$ and has the area property, then
    $\rho(f)=\rho(g)$ for every function $g$ that is topologically
     equivalent to $f$.
\end{prop}
\begin{proof}
 Since $g\circ\phi = \psi\circ f$, and $f$ and $g$ are holomorphic, the
  dilatation of $\phi$ is supported on $f^{-1}(K)$.  
   By the area property, this set has finite cylindrical area.
   Hence the Teichm\"uller-Wittich
     Theorem and \ref{prop:topequ} \ref{item:asconforder} imply 
    that indeed $\rho(f)=\rho(g)$.

   If furthermore $f\in\classS$, then by Observation \ref{obs:isotopy} there is
    a quasiconformal homeomorphism $\tilde{\psi}$ that is isotopic to $\psi$ relative $S(f)$
    and whose dilatation is supported away from the singular values.
    According to Proposition \ref{prop:topequ} 
    \ref{item:isotopy}, there is $\tilde{\phi}$ such that
     $\tilde{\psi}$ and $\tilde{\phi}$ are witnessing homeomorphisms for
    the quasiconformal equivalence of $f$ and $g$. 
    So $\tilde{\phi}$ is asymptotically conformal and $\rho(f)=\rho(g)$, 
    as claimed.

  It remains to show that the map $g$ also has the area property. This follows from a geometric fact 
   concerning quasiconformal mappings, which we state separately as Proposition \ref{prop:areapropertyqc} below
   for future reference.
\end{proof}
\begin{prop}[Quasiconformal mappings and the area property] \label{prop:areapropertyqc}
  Suppose that $\phi$ is a quasiconformal mapping, and that the dilatation of $\phi$
    is supported on a set $A\subset\C\setminus\D$ of finite cylindrical area. 
   Then also $\areacyl(\phi(A)\setminus\D)<\infty$. 
 \end{prop} 
 \begin{proof}
   This claim is related to the area distortion problem for quasiconformal mappings,
   which was solved completely by 
   Astala \cite{astaladistortion}. The estimates we require are essentially due 
   to Gehring and Reich \cite{gehringreich}.   Instead of
   proving our claim directly
   using these methods, we shall formally derive
   it from the following result stated in Astala's article \cite[Lemma 3.3]{astaladistortion}.
   \emph{For every $K$, there is 
   a constant $C$ with the following property. If $f\colon \D\to\D$ is a $K$-quasiconformal
    homeomorphism with $f(0)=0$
   whose dilatation is supported on a closed subset $E\subset\D$, then 
    $\area(f(E)) \leq C\cdot \area(E)$.}

  To prove our claim, let $A_k$, for $k\geq 0$, denote the 
    the intersection of $A$ with the annulus $\{2^k < |z| < 2^{k+1}\}$, and set 
   $\lambda_k \defeq \areacyl(A_k)$.
  By assumption,
       \[ \sum_{k=0}^{\infty} \lambda_k < \infty. \]
   Now consider $\tilde{A}_k \defeq \phi(A_k)$ and its cylindrical area $\tilde{\lambda}_k$. We
     must show that the sequence $\tilde{\lambda}_k$ is also summable. 

    Note that the conclusion of the claim
    does not change under post-composition of $\phi$ by affine functions. 
    Hence we can assume that $\phi(0)=0$, and 
     (by the Teichm\"uller-Wittich-Belinski-Lehto Theorem) that
     $\phi(z)\sim z$ as $z\to\infty$. 

   For each $k\geq 0$, define a map
    $\phi_k$ by 
      \[ \phi_k(z) \defeq \frac{\phi(2^{k+1}z)}{2^{k+1}}. \]
     Also let $\psi_k$ be a Riemann map for $\phi_k(\D)$; i.e., let 
      $\psi_k\colon  \phi_k(\D)\to \D$ be a conformal isomorphism with $\psi_k(0)=0$ and 
      $\psi_k'(0)>0$. Since $\phi(z)\sim z$ as $z\to\infty$, we see that
      $\phi_k$ converges uniformly to the identity. We define a quasiconformal map $f_k\colon\D\to\D$ by
      \[ f_k \defeq \psi_k \circ \phi_k. \]

     Set $E_k \defeq \{z\in \D\colon  2^{k+1}z\in A\}$; i.e. $E_k$ is the support of the dilatation of $f_k$.  Since cylindrical area is invariant under
       linear maps, we have 
      \[ \area(E_k) \asymp \sum_{j=0}^k \frac{\lambda_j}{2^{k-j}} \]
      (for a constant independent of $k$). For the same reason, we have 
        \[ \tilde{\lambda}_k  \asymp \area(\phi_{k+1}(A_k/2^{k+2})), \qquad\text{and hence}
             \qquad \tilde{\lambda}_k  \asymp \area(f_{k+1}(A_k/2^{k+2})) \]
     by the Koebe Distortion Theorem. (Observe that $A_k/2^{k+2}$ is contained in the
      disc of radius $1/2$ around the origin, and hence $\phi_{k+1}(A_k/2^{k+2})$ is well inside
      $\phi_{k+1}(\D)$. So we can indeed apply the Distortion Theorem to the map 
      $\psi_{k+1}^{-1}$ on $f_{k+1}(A_k/2^{k+2})$.) 
    As $A_k/2^{k+2}\subset E_{k+1}$ by definition, it follows that 
 \[ \tilde{\lambda}_k \leq \const \cdot \area( f_{k+1}(E_{k+1})). \]
       By Astala's result stated above, we thus see that
       \[ \tilde{\lambda}_k \leq \const\cdot \sum_{j=0}^{k+1} \frac{\lambda_j}{2^{k+1-j}}. \]
       Hence
\[           \sum_{k=0}^{\infty} \tilde{\lambda}_k \leq
        \const\cdot \sum_{k=0}^{\infty} \sum_{j=0}^{k+1} \frac{\lambda_j}{2^{k+1-j}} 
    \leq \const\cdot \sum_{j=0}^{\infty} \lambda_j \sum_{m=0}^{\infty} \frac{1}{2^m} 
         < \infty. \qedhere \]
 \end{proof}
\begin{remark}
  Astala states his lemma for \emph{closed} subsets of the disc, but appears to prove it only when the set $E$ is \emph{compact}. 
    Since his estimates depend only on $K$, the version for closed subsets can be reduced to the compact one. Alternatively, for each $k$ we can solve the
    Beltrami equation to obtain a map $\sigma_k$ whose dilatation agrees with that of $\phi$ for $|z|<2^k$ and is zero otherwise. It is easy to see that 
    $\sigma_k\to \phi$ (since the corresponding dilatations converge almost everywhere), 
   and $\area(\sigma_k(A_j))\to \tilde{\lambda}_j$ as $k\to\infty$ for all $j$
   . We can then easily obtain the desired conclusion by applying the 
    proof as above, replacing $\phi$ by $\sigma_k$ in the definition of $f_k$; then the dilatation 
   of $f_k$ has compact support. 
    Using the fact that all estimates are uniform, we easily obtain the desired conclusion. 
\end{remark}

\subsection*{Some equivalent formulations of the area property}
 We now discuss some formulations of the area property that
  are easy to verify. We begin with an infinitesimal version:
\begin{prop}[Infinitesimal area property] \label{prop:equivalentformulation}\label{prop:infinitesimal}
  A transcendental entire function $f$ has the area property if and only if 
\begin{equation}
\sum_{z\in f^{-1}(w)\setminus\D}\frac{1}{|z|^{2}|f^{\prime }(z)|^{2}}<\infty
\label{1overzf}
\end{equation}
   for all $w\in \C\setminus S(f)$.

   Furthermore, if \eqref{1overzf} holds for some $w_0\in\C\setminus S(f)$, then it also holds for all $w$ that belong to the same component of
     $\C\setminus S(f)$ as $w_0$. 
\end{prop}
\begin{proof}
  Let $w\in S(f)$, and let $D=D_w\subset \C\setminus S(f)$ be a closed topological disc whose 
  interior contains $w$. 
  Let 
  $\tilde{D} \subset \C\setminus S(f)$ 
   be a slightly larger simply-connected domain
   with $D\subset \tilde{D}$. 

  Let $z\in f^{-1}(w)\setminus\D$, and let $V_z$ be the component of $f^{-1}(D)$ containing $z$. 
   If $\tilde{V}$ is the component of $f^{-1}(\tilde{D})$ containing $V_z$, then 
    $f:\tilde{V}\to \tilde{D}$ is a conformal isomorphism, and if $D_w$ was chosen sufficiently
    small, then $V$ does not intersect the disc of radius $1/2$ around the origin. It follows that
      \begin{equation}\label{eqn:keyinfinitesimalstep}
         \min_{\zeta\in V}\frac{1}{|\zeta|^2 |f^{\prime}(\zeta)|^2}
           \asymp \areacyl(V_z) \asymp 
                \max_{\zeta\in V}\frac{1}{|\zeta|^2 |f^{\prime}(\zeta)|^2} \end{equation}
    by the Koebe Distortion Theorem. In particular, the area property implies~\eqref{1overzf}. 

  For the ``if'' direction, suppose that \eqref{1overzf} holds and $K\subset S(f)$ is an arbitrary 
   compact set.  Then we can cover $K$ by finitely many discs $D_{w_1},\dots,D_{w_k}$ as 
   above, and it follows that $\areacyl(f^{-1}(K)\setminus\D)<\infty$.  
   Furthermore,~\eqref{eqn:keyinfinitesimalstep} shows that~\eqref{1overzf} is 
   an open and closed property, and hence the final claim follows.
\end{proof}

 The preceding proof relies crucially on the Koebe Distortion Theorem.
    It is well-known that area distortion theorems hold also for branched covering maps of bounded degree. This allows us to deduce
    that the area property will hold not only near regular values (as in Definition \ref{defn:areaproperty}), but also near non-asymptotic critical
    values for which the degree of the critical points is bounded. In particular, this justifies the remark after Conjecture
    \ref{conj:eremenkolyubich}. 
\begin{lem}[Bounded criticality]\label{lem:boundedcriticality}
  Let $f$ be a transcendental entire function. Let $s$ be an isolated point of $S(f)$ that is not an asymptotic value and such that
   the local degree of $f$ near any preimage of $s$ is uniformly bounded by a constant $\Delta$.

  Let $D$ be a round disc around $s$ such that $\overline{D}\cap S(f)=\{s\}$. If the condition~\eqref{1overzf} holds for all 
     $z\in D^* \defeq D\setminus\{s\}$, then $f^{-1}(D)\setminus\D$ has finite cylindrical area. 
\end{lem}
\begin{proof}
  Let $\tilde{D}$ be a slightly larger round disc around $D$ whose closure still does not intersect the singular set except in $s$.
    By postcomposing with an affine map, we may assume for convenience that $\tilde{D}=\D$. 

 Let $\tilde{V}$ be a component of $f^{-1}(\tilde{D})$ that does not intersect $\D$. 
  The assumptions imply that $\tilde{V}$ is simply-connected and contains a unique preimage $c$ of $s$, of some degree $d\in\{1,\dots,\Delta\}$. 
   Let $\phi:\D\to \tilde{V}$ be a conformal isomorphism with $\phi(0)=c$. It follows that $f(\phi(z)) = \theta\cdot z^d$ for some $\theta\in\C$ 
    with $|\theta|=1$ and
    all $z\in\D$; by precomposing $\phi$ with a rotation we can assume that $\theta=1$. 

  Let $r<1$ denote the radius of $D$, so that $D=B_r(0)$ (where we use the standard notation for Euclidean balls).
    Hence $V_c\defeq \phi(B_{r^{1/d}}(0))$ is the component of $f^{-1}(D)$ containing $c$. 
    We may assume without loss of generality that $r>1/2$. Set $z \defeq \phi(2^{-1/d})\in V_c$; then $f(z)=1/2$. By the functional equation,
    we have $|f'(z)|\cdot|\phi'(2^{-1/d})|=d\cdot 2^{-(d-1)/d}$.  By the Koebe Distortion Theorem, it follows that
   \[
        \areacyl(V_c) = \areacyl(\phi(B_{r^{1/d}})) \asymp 
            \frac{|\phi'(2^{-1/d})|}{|\phi(2^{-1/d})|} =
            \frac{d}{2^{\frac{d-1}{d}}\cdot |z|\cdot |f'(z)|} \asymp \frac{1}{|z|\cdot |f'(z)|}.
    \]
   (Here the constants depend on $\Delta$, but not otherwise on $f$. In particular, they are independent of the choice of $\tilde{V}$.)
  
  So the total logarithmic area of all of these preimages {$V_c$} is bounded in terms of the sum~\eqref{1overzf} for $w=1/2$. 
   {It remains to show that the part of $f^{-1}(D)\setminus\D$ that 
     is contained in preimage
    components of $\tilde{D}$ that do intersect the unit disc has finite area. But each such
    component is bounded, and hence has finite area. Furthermore, by 
    the local mapping properties of holomorphic functions, the number of components of $f^{-1}(\tilde{D})$ is locally finite,
    and hence there are only finitely many components that intersect $\overline{\D}$. The claim follows.}
\end{proof}

There are various other ways to reformulate the area property. For example, since 
  $f$ is a covering map on every component of 
    $f^{-1}(\C\setminus S(f))$, the derivative $f'(z)$ can be expressed in terms of the hyperbolic metric of $f^{-1}(\C\setminus S(f))$. The
    hyperbolic metric of simply-connected domains is particularly easy to estimate in terms of the distance to the boundary, and hence
    we obtain the following.

 \begin{prop}[Distances and the area property] \label{prop:distance}
  Let $f$ be a transcendental entire function, and let $w\in \C\setminus S(f)$. Let $K\subset\C\setminus\{w\}$ be a closed connected set with
    $S(f)\subset K$ and $\# K >1$. Then
    \eqref{1overzf} holds if and only if
     \[  \sum_{z\in f^{-1}(w)\setminus \D}\frac{\dist(z,f^{-1}(K))^2}{|z|^2} < \infty.\]
 \end{prop}
 \begin{proof}
   Let $z\in f^{-1}(w)$, let $W$ be the component of $\C\setminus K$ containing $w$ and let $V$ be the component of $f^{-1}(W)$ containing $z$.
 Then $f\colon V\to W$ is a holomorphic covering map.  If $\rho_V$ and $\rho_W$ denote the densities of the hyperbolic
    metrics of $V$ and $W$, we thus have
     \begin{equation}\label{eqn:coveringderivative} |f'(z)| = \rho_V(z)/\rho_W(w). \end{equation}

 The domain $W$ is either simply-connected or conformally equivalent to the 
   punctured unit disc 
    (if $K$ is bounded and $W$ is the unbounded connected component of $\C\setminus K$). The only covering spaces of the punctured disc are given by
   the universal covering (via the exponential map) and the punctured disc
   (via $z\mapsto z^d$, $d\geq 1$). The latter case cannot occur in our
   setting, since $f$ is transcendental; so we see that 
   $V$ is simply-connected.
  The claim now follows from~\eqref{eqn:coveringderivative} and the standard estimate~\eqref{eqn:standardestimate}. 
 \end{proof}

\subsection*{A return to Poincar\'e functions}
 We now study the area property for Poincar\'e functions, 
  proving Theorem \ref{thm:linearizers}.

\begin{thm}[Area property for linearizers]\label{thm:arealinearizers}
  Let $p$ be a polynomial of degree $\geq 2$
   with a repelling fixed point at $0$, and let 
   $f\colon \C\to\C$ be the Poincar\'e function for this fixed point.

  Let $w\in \C\setminus \P(p) = \C\setminus S(f)$. Then 
    (\ref{1overzf}) holds for $w$ if and only if $w$ does not belong to a Siegel disc of $p$.
\end{thm}
\begin{proof}
  It suffices to prove the theorem for $w\in F(p)$. Indeed, if $w\in J(p)$,
   then we can let $w'$ be a point in the basin of infinity of $p$ that
   belongs to the same component of $\C\setminus S(f)$ as $w$. 
  {By
   Proposition \ref{prop:infinitesimal}, 
   property~\eqref{1overzf} holds for $w'$ if and only if it holds for $w$.}

 Let $\eta<1$ be small enough so that 
   $f$ is injective on a neighbourhood of the closed disc of radius $\eta$ around $0$.
    We define 
    \[ A \defeq \{z\in\C\colon  \eta/|\lambda| < |z| < \eta \}, \]
 where $\lambda=p'(0)$. For simplicity, we may assume that $\eta$ is chosen such that 
  $f(\partial A)$ does not intersect the backwards orbit
    $\bigcup_{n=0}^{\infty} p^{-n}(w)$. 

 Suppose that $z\in f^{-1}(w)\setminus\overline{\D}$, and let $n\geq 1$ be minimal such that
   $|\lambda|^n \geq |z|/\eta$. Set $\tilde{z} \defeq z/\lambda^n$ and $\tilde{w} \defeq f(\tilde{z})$. By the functional relation 
    $f(\lambda z) = p(f(z))$, we have 
\begin{align*}
          p^n(f(\tilde{z})) &= f(\lambda^n\tilde{z}) = f(z) = w \qquad\text{and}\\
       |z|\cdot |f'(z)| &= |z|\cdot \frac{|(p^n)'(\tilde{w})| \cdot |f'(\tilde{z})|}{ \lambda^n} =
           |\tilde{z}| \cdot |f'(\tilde{z})| \cdot |(p^n)'(\tilde{w})|. \end{align*}
   In particular, by our assumption on $\eta$, we have $\tilde{z}\in A$, and hence the
    numbers $|\tilde{z}|$ and $|f'(\tilde{z})|$ are uniformly bounded away from 
    $0$ and $\infty$. Furthermore, 
      {for every $n\geq 1$, 
the correspondence between $z$ and $\tilde{w}$
    defines a bijection between the points of $f^{-1}(w)$ of modulus between
     $|\lambda^{n-1}|$ and $|\lambda^n|$ and the intersection
     $p^{-n}(w)\cap f(A)$.} So, for $N\geq 1$, 
    \begin{equation}\label{eqn:poincareseries}
        \sum_{z\in f^{-1}(w), 1\leq |z|\leq \lambda^N} \frac{1}{|z|^2|f'(z)|^2} \asymp
           \sum_{n=1}^{N}\sum_{\tilde{w}\in f(A)\cap p^{-n}(w)}\frac{1}{|(p^n)'(\tilde{w})|^2}.
    \end{equation}

  If $w$ does not belong to a Siegel disc, then it lies in the
    basin of infinity of $p$, an atttracting or parabolic basin, or a Fatou component that is not periodic. In each case, we can find
   a small disc $D$ around $w$ such that $p^{-n}(D)\cap D=\emptyset$ for all $n\geq 1$.  
{So} the sum
   \[ \sum_{n=0}^{\infty} \sum_{\tilde{w}\in p^{-n}(w)} \frac{1}{(p^n)^{\sharp}(\tilde{w})^2} \]
   (the \emph{Poincar\'e series} at exponent $2$) is comparable to the spherical area of the 
   backward orbit of the disc $D$ under $p$, {and hence
   finite.} 
   (Here we use 
    $(p^n)^{\sharp}$ to denote the derivative of $p^n$ as measured with respect to the spherical metric, both in the range and in the domain.)
    Since the spherical metric and the Euclidean metric are comparable on the bounded set $f(A)$, we see from~\eqref{eqn:poincareseries} that $f$ satisfies~\eqref{1overzf}. 

 On the other hand, suppose that $w$ belongs to a Siegel disc $U$ of $p$ of period $n$. 
   Since $p^n|_U$ is conjugate to an irrational rotation, there is a
   sequence $n_k$ such that $p^{n_k}|_U\to \id$. 

  Set $w_k \defeq (p^{n_k}|_U)^{-1}(w)$; then $w_k\to w$. 
   Fix $\zeta_0\in f^{-1}(w)$ and let $D$ be a neighbourhood of $\zeta_0$ on which $f$ is
     injective. By disregarding finitely many entries, we can ensure that $w_k\in f(D)$ for all $k\geq 1$. Let us define
      $\zeta_k \defeq (f|_D)^{-1}(w_k)$ and $z_k \defeq \lambda^{n_k}\cdot \zeta_k$; we may assume that $|z_k|\geq 1$ for all $k$. Then
      $f(z_k)=w$ and 
  \[ z_k\cdot f'(z_k) =
        \zeta_k \cdot f'(\zeta_k) \cdot (p^{n_k})'(w_k) \to
           \zeta_0 \cdot f'(\zeta_0). \]
  Thus
  \[
       \sum_{z\in f^{-1}(w)\setminus\D}\frac{1}{|z|^{2}|f^{\prime }(z)|^{2}} 
 \geq \sum_{k=1}^{\infty} \frac{1}{|z_k|^{2}|f^{\prime }(z_k)|^{2}} = \infty, \]
   as required.
\end{proof}

{We conclude the section by including a counterexample to Conjecture~\ref{conj:eremenkolyubich}.}

\begin{thm}[Poincar\'e functions of postcritically finite hyperbolic polynomials]\label{thm:eremenkolyubichcounterexample}
{
  Let $p$ be a polynomial such that every critical point of $p$ eventually maps to a superattracting cycle. 
    Let $z_0$ be a fixed point of $p$ that does not belong to the boundary of an invariant Fatou component, and let $f\in\classS$ be the corresponding Poincar\'e 
    function. Then $f$ has no asymptotic values, but nonetheless satisfies~\eqref{eqn:eremenkolyubich}.}
\end{thm}
\begin{remark}
{
  As an example, one can take $p(z)=z^2-1$ and $z_0=(1+\sqrt{5})/2$. This Poincar\'e function was previously considered in 
    \cite[Appendix~B]{helenaconjugacy} as an example of a function $f\in\classS$ having no asymptotic values but critical points of arbitrarily high order.}
\end{remark}
\begin{proof}
{
  By \cite[Corollary~4.4]{helenajoern}, the function $f$ has no asymptotic values. Since the filled Julia set $K\defeq K(p)$ has non-empty interior, it has
   positive area. Now $f^{-1}(K)$ is completely invariant under multiplication by the multiplier $\lambda$ of $z_0$, by the 
   functional equation~\eqref{eqn:functionalequation}. Hence, for all $n\geq 1$, 
    \[  \int_{z\in\C: 1\leq |z|\leq |\lambda|^n \text{ and } f(z)\in U } \frac{\dif x \dif y}{|z|^2} = 
        n\cdot \int_{z\in\C: 1\leq |z| \leq |\lambda| \text{ and }f(z)\in U} \frac{\dif x \dif y}{|z|^2} =:\eps\cdot n, \]
    with $\eps>0$. Hence~\eqref{eqn:eremenkolyubich} holds, as required.}
\end{proof}

\section{A counterexample to the area property near infinity}\label{sec:counterexample}

\begin{proof}[Proof of Theorem \ref{thm:counterexample}]
 We shall now show that there exists an entire function $f\in\B$ that violates the area property near infinity, using a construction 
   from \cite[Section~7]{approximationhypdim}. As indicated in the introduction, 
  {and stated in 
    \cite{approximationhypdim} (following Theorem 1.11 in that paper)}, 
   it can be shown that the 
   exact function considered in that paper, which is a hyperbolic entire function with full hyperbolic dimension, also violates
   the area property. However, it shall be slightly more convenient for us to use the same construction, but with different
   parameters. {We shall remark on the original example following the completion of the proof.}

 As in \cite{approximationhypdim}, the proof proceeds in two steps:
  \begin{itemize}
    \item First, a simply-connected domain $V$ is constructed that does not intersect its 
      $2\pi i$-translates, along with a conformal isomorphism 
    \[ G\colon  V\to  H \defeq \{x+iy \colon  x > - 14\log_+|y|\}, \]
      where again $\log_+|y| = \max(0,\log|y|)$. The tract is chosen such that the (not globally defined)
      function $g\colon\exp(V)\to \C$; $g(\exp(z)) = \exp(G(z))$ does not have the area property near
      infinity.
   \item By \cite[Theorem 1.7]{approximationhypdim}, 
     the function $g$ can be approximated by a transcendental entire function $f\in\B$ with
      $|f(z)-g(z)| = O(1/|z|)$ for $z\in \exp(V)$, and $|f(z)|=O(1/|z|)$ elsewhere. 
       It can then be checked that $f$ also does not have the area property
       near infinity. 
 \end{itemize}

 For the first step, we shall use \cite[Lemma 7.2 and Theorem 7.4]{approximationhypdim}, which 
   imply that the domain $V$ and the function $G$ can be chosen such that the following 
   properties hold.
 \begin{enumerate}[\normalfont(1)]
  \item $V\subsetneq\{ a + ib\colon  a>1 \text{ and }|b|<\pi\}$;\label{item:stripV}
  \item there are $C_1>1$ and $Q\geq 1$ such that 
       $\re z /C_1 \leq \log_+ |G(z)| \leq C_1\re z$ for all $z\in V$ with $\re z \geq Q$.
 \label{item:growthofG}
  \item there exist constants $C_2>0$, $k_0\geq Q/2\pi$ 
     and a collection of points $(\zeta_k)_{k\geq k_0}$ in $V$ such that
     $\re \zeta_k = 2\pi k$, $\dist(\zeta_k,\partial V)\geq C_2$ 
    and $\re G(\zeta_k) = 1$ for all $k\geq k_0$.\label{item:zetak}
 \end{enumerate}

We remark that $V$ consists of a 
   central strip of fixed width, to which a sequence of ``side chambers'' are attached;
   see \cite[Figure 1]{approximationhypdim}. 
  These are equally spaced in a $2\pi$-periodic manner, and
   the $k$-th chamber is connected to the central strip by a small opening of size $\eps_k$. The 
   points $\zeta_k$ are precisely the mid-points of these chambers, and the opening size $\eps_k$ 
    is chosen such that $\re G(\zeta_k)=1$. This ensures property
   \ref{item:zetak}. Property \ref{item:growthofG} is a simple consequence of the description of the tract and
   classical geometric function theory. For details, we refer to \cite[Section 7]{approximationhypdim}.

 Let us verify that, when $G$ is chosen with properties \ref{item:stripV} to \ref{item:zetak}, we have 
    \begin{equation}\label{eqn:areaG} 
      \sum_{m\in\Z} |(G^{-1})'(1 + 2\pi i m)|^2 = \infty. 
     \end{equation}
  As in Proposition~\ref{prop:infinitesimal}, the formula~\eqref{eqn:areaG} implies via Koebe's Distortion Theorem
     that $g$ does not have the area property
     near infinity. (By this, we mean that $g^{-1}(K)$ has infinite cylindrical area for any
     compact set $K\subset \C\setminus\overline{\D}$ with nonempty interior.)

\begin{claim}[Claim 1]
 There is a constant $C_3>1$ such that $k/C_3\leq \dist(G(\zeta_k),\partial H)\leq C_3\cdot k$ for
   all $k\geq k_0$. 
\end{claim}
\begin{subproof}
  By~\ref{item:zetak} above, we have $\re \zeta_k = 2\pi k\geq Q$, 
   and hence $\log_+ |G(z)| \leq 2\pi C_1 k$ by~\ref{item:growthofG}. On the other hand,
   $\dist(\zeta,\partial H)\leq 15\log_+\im\zeta \leq 15 \log_+|\zeta|$ whenever
   $\zeta\in H$ with $\re\zeta = 1$.
   This implies the upper bound; the lower bound follows analogously.
\end{subproof}

Since $G$ is a conformal isomorphism, we see from the standard estimate~\eqref{eqn:standardestimate}, together with 
    Claim 1 and~\ref{item:zetak}, that
    \begin{equation}\label{eqn:zetakderiv}
    |G'(\zeta_k)| = \frac{\rho_V(\zeta_k)}{\rho_H(G(\zeta_k))} \leq 
         4\cdot \frac{\dist(G(\zeta_k),\partial H)}{\dist(\zeta_k, \partial V)} \leq 
       \const\cdot k. \end{equation}

\newcommand{\M}{\mathcal{M}}
Now, for $k\geq k_0$, consider the set $\M_k\defeq\{m\in\N: |2\pi m - \im G(\zeta_k)|\leq k/(5C_3)\}$. By~\eqref{eqn:zetakderiv} and Koebe's
  Distortion Theorem, we have
 \[ |(G^{-1})'(1 + 2\pi i m) \geq \frac{\const}{k} \]
 when $m\in \M_k$. Also note that
  \[ \# \M_k \geq \left\lfloor \frac{k}{5\pi C_3}\right\rfloor \geq  \const\cdot k\]
 for all $k\geq k_1 \defeq \max(\lceil5\pi C_3\rceil,k_0)$ (for a constant independent of $k$). 

\begin{claim}[Claim 2]
  The sets $\M_k$ are pairwise disjoint.
\end{claim}
\begin{subproof}
 Consider the vertical line segment $S$ connecting the points $G(\zeta_k)-ki/(5C_3)$ and
    $G(\zeta_k)+ki/(5C_3)$.
 By Claim 1, we have 
    \[ \dist( \zeta , \partial H)\geq  \frac{4k}{5C_3} \] 
  for all $\zeta\in S$. By the standard estimate~\eqref{eqn:standardestimate}, 
   we see that the hyperbolic distance in $H$ between the midpoint  $G(\zeta_k)$ and any point of 
   $S$ is at most $1/2$. Recall that $1+2\pi m\in S$ for all $m\in \M_k$, by definition.

On the other hand, the 
   hyperbolic distance in $V$ between $\zeta_k$ and $\zeta_{k'}$, with $k\neq k'$, is strictly
   larger than $1$. Indeed, the hyperbolic distance between $\zeta_k$ and $\zeta_{k'}$ 
    in the strip 
   $\{a+ib \colon |a|<\pi\}\supset V$ is strictly larger than one by direct computation, 
   and the claim follows from the comparison principle of hyperbolic geometry. 
\end{subproof}

 Combining all the estimates, we see that
    \begin{align*}
           \sum_{m\in\Z} |(G^{-1})'(1 + 2\pi i m)|^2 &\geq
           \sum_{k=k_1}^{\infty} \sum_{m\in \M_k} |(G^{-1})'(1 + 2\pi i m)|^2 \\ &\geq
         \const \cdot \sum_{k=k_1}^{\infty} \#\M_k \cdot \frac{1}{k^2} \geq
         \const \cdot \sum_{k=k_1}^{\infty}\frac{1}{k} = \infty. 
   \end{align*}
 This establishes~\eqref{eqn:areaG}.

Now let $f\in\B$ be the entire function satisfying 
    $|f(z)-g(z)|=O(1/|z|)$ for $z\in \exp(V)$, whose existence is guaranteed by
    \cite[Theorem 1.7]{approximationhypdim}. It is not difficult
    to verify, using the above information about the construction, that 
    $f$ does not have the area property near infinity. Instead, we shall derive this fact
    using quasiconformal equivalence. By \cite[Theorem 1.8]{approximationhypdim}, there is $R>0$ and 
    a quasiconformal map
    $\phi\colon \C\to\C$ such that $g(\phi(z))=f(z)$ whenever $|f(z)| > R$. Moreover, by
    \cite[Theorem~6.3]{approximationhypdim},
    the map 
    $\phi$ can be chosen such that $\phi(z)=z$ on $f^{-1}(D)$, where $D$ is a simply-connected 
   neighbourhood
    of $S(f)$. 

  In particular, the dilatation of $\phi$ is supported on the set
    \[ \{z\in\C\colon  |f(z)|\leq R\text{ and }f(z)\notin D\}. \] 
   If $f$ had the area property near $\infty$, 
    then this set would have finite cylindrical area, and hence
   $\phi$ would preserve the property of having finite cylindrical area by 
   Proposition
    \ref{prop:areapropertyqc}. But then $g$ also has the area property (due to
   the functional relation $g(\phi(z))=f(z)$), which is a contradiction.
\end{proof}
\begin{remark}
{%
  The function $G$ constructed in 
   \cite[Section~7]{approximationhypdim} has the same properties as above, except
    that, in~\ref{item:zetak}, the points $\zeta_k$ satisfy
    $\re G(\zeta_k)\asymp k$, rather than $\re G(\zeta_k) = 1$. (This is what enables the 
    construction of a suitable iterated function system, which gives rise to the desired
    hyperbolic sets.) Moreover, this is not required to hold for \emph{all}
    sufficiently large $k$ (although that can be arranged), but only for those satisfying
    $K_i < 2\pi k + R_0 < 3K_i$ for some $i$, where $(K_i)$ is a (possibly rapidly) increasing
   sequence and $R_0$ is a universal constant; see \cite[Corollary~7.5]{approximationhypdim}. Without loss of generality, let us assume that
   $K_0\geq \max(2\pi, R_0)$.}

{%
 To see that this function also satisfies~\eqref{eqn:areaG}, and hence violates the area property near infinity, define a sequence
   $(\tilde{\zeta}_k)$ by $\tilde{\zeta}_k \defeq G^{-1}(1 + i\cdot \im G(\zeta_k))$. By~\ref{item:growthofG}, we see that
   $\log_+ \im G(\zeta_k) \geq \const\cdot k$, and hence the hyperbolic distance (in $H$) between $\zeta_k$ and $\tilde{\zeta}_k$ is
   bounded from above, independently of $k$. Thus $\dist(\tilde{\zeta}_k,\partial V)\geq \const$, independently of $k$, and we 
   can apply the same argument as above, replacing $\zeta_k$ by $\tilde{\zeta}_k$. We conclude that indeed
    \begin{align*} \sum_{m\in\Z} |(G^{-1})'(1+2\pi im)|^2 &\geq
         \sum_{i\geq 0} \sum_{k=\lfloor (K_i-R_0)/2\pi\rfloor +1}^{\lceil (3K_i-R_0)/2\pi\rceil -1} \frac{\const}{k}  \\
       &\geq \const\cdot \sum_{i\geq 0} \left\lfloor\frac{2K_i}{2\pi}\right\rfloor\cdot \frac{1}{3K_i-R_0} \geq
          \const\cdot \sum _{i\geq 0} \frac{K_i}{2K_i} = \infty. \end{align*}}
  \end{remark}

\section{Quadratic differentials}\label{sec:qd}
 In this section, we discuss a matter that is closely connected to 
   the area property, namely the behaviour of \emph{quadratic differentials} under
   the pushforward by an entire function. Recall that a \emph{quadratic differential} is a tensor of the form
   $q(z)\dif z^2$ (in local coordinates). {If the local coefficient $q$ can always be chosen to be
   measurable, holomorphic or meromorphic, then the differential itself is called measurable, holomorphic or meromorphic. 
   A \emph{pole} of a meromorphic quadratic differential is then a point near which the local coefficient must have a pole; observe that locally such
   a differential can always be written in the form $\dif z^2/z^d$, where $d$ is the \emph{order} of the pole. Any quadratic differential gives
   rise to an associated area form $|q(z)|\,\modd{z^2}$; the total area 
              \[ \int |q(z)|\,\modd{z^2} \]
   is referred to as the \emph{(total) mass} of the differential. 
   By an elementary calculation, the mass of a meromorphic quadratic differential is finite near a simple pole, but infinite near a pole of higher order.}

 {Quadratic} differentials play a key role in complex dynamics 
   and complex
   analysis. Of particular interest is the \emph{pushforward} operator: given an analytic map
   $f\colon U\to V$, and a measurable (in the following usually meromorphic or even holomorphic)
   quadratic differential $q = q(z)\dif z^2$ on $U$, 
    its pushforward is defined to
    be the formal sum 
    \[ f_*q \defeq \left(\sum_{f(z)=w} \frac{q(z)}{f'(z)^2}\right)\dif w^2 . \]
   Of course, in general this sum may or may not converge; if it converges absolutely, we shall say that $q$ is \emph{$f$-summable}. 
    
  If $q$ is a meromorphic
   quadratic differential on the Riemann sphere with at most finite poles, then $q$ has
    finite total \emph{mass}. Since the pushforward under a holomorphic map can never increase mass (and may in fact decrease it 
   due to possible cancellations), it follows that such a 
    differential is always summable. On the other hand, for quadratic differentials with at worst double poles the total mass is infinite, so this argument {cannot be
    used to show that the differential is $f$-summable}. 
  {(Quadratic differentials with double poles can play an important role in holomorphic dynamics;
   for example, they appear when 
    bounding the number of non-repelling cycles, see \cite{fatoushishikura}.)} However, if
    $f$ is entire and has the area property, then 
    the pushforward converges absolutely at least for $w\notin S(f)$. 
    If furthermore $f\in\B$, then we can say more about the nature of the singularity of this pushforward
     near $\infty$. This is a consequence of the following observation.

 \begin{prop}[Quadratic differentials and logarithmic singularities]\label{prop:doublepole}
   Let $V\subsetneq\C$ be an unbounded simply-connected domain, and let 
     $f:V\to\C\setminus\overline{\D}=: W$ be a holomorphic universal covering map. 
      Suppose furthermore that 
     \begin{equation}\label{eqn:areatracts}
       \sum_{z\in f^{-1}(w)\setminus\D} \frac{1}{|z|^2\cdot|f'(z)|^2} < \infty 
      \end{equation}
    for all $w\in W$. 

   Let $q=q(z)\dif z^2$ be a holomorphic quadratic differential on $V$ for which there is 
     a constant $C>0$ such that
    $|q(z)| \leq  C/|z|^2$ for all sufficiently large $z\in V$. Then
    $q$ is $f$-summable for all $w\in W$ and the pushforward
     $f_*q$ has at most a double pole at $\infty$. 
 \end{prop} 
 \begin{proof}
   For convenience,  we may assume that
    $0\notin \cl{V}$ (restricting $V$ to the complement of a slightly larger disc and reparametrizing, if necessary). This ensures
     that we can write $q=(\rho(z)/z^2)\dif z^2$ for all $z\in\V$, with $\rho(z)\leq \tilde{C}$ for a suitable constant $\tilde{C}>0$. 
     The pushforward is then the formal sum
    \[ f_* q = \left(\sum_{z\in f^{-1}(w)}\frac{\rho(z)}{f'(z)^2 z^2}\right) \dif w^2 =: \sigma(w)\dif w^2. \] 
   By virtue of the area property~\eqref{eqn:areatracts} for $f$, the sum is absolutely convergent
   for $w\in W$, so $\sigma$ is defined and holomorphic
    on $W$. (Recall that the sum~\eqref{eqn:areatracts} converges locally uniformly in $w$ due to the Koebe Distortion Theorem.) We must show that 
    \[ |\sigma(w)| = o(1/|w|)\qquad \text{as $w\to 0$.} \]

  To begin, let us perform a \emph{logarithmic change of variable} in the sense of
    Eremenko and Lyubich \cite{alexmisha} as follows. 
    Let $U$ be a connected component of $\exp^{-1}(V)$. Since $f\circ\exp:U\to W$ is
    a universal covering map, there is a conformal isomorphism $\phi$ from     $\HH \defeq \{a+ib \colon a>0\}$ to $U$
    such that $f(\exp(\phi(\zeta))) = \exp(\zeta)$ for all $\zeta\in\HH$. 

  Let $\zeta\in\HH$ with $\re \zeta > 1$, and define $\tilde{\zeta} \defeq 1 + \im \zeta$.    
  By Koebe's Distortion Theorem, applied to the restriction of $\phi$ to the 
   disc of radius $\re \zeta$ around $\zeta$, we see that 
    \[ |\phi'(\zeta)| \leq 8\re \zeta\cdot |\phi'(\tilde{\zeta})|. \] 

 Now let $w\in W$ with ${|w|>e}$, and let $\zeta_0\in \exp^{-1}(w)$. We set
  $\tilde{w} \defeq e\cdot w/|w|$, so that $|\tilde{w}| =e$. 
  Differentiating the relation $f \circ \exp\circ \phi = \exp$, we see that 
        \[ f'(z) \cdot z  = \frac{e^{\zeta}}{\phi'(\zeta)}, \]
    {whenever} $z=e^{\phi(\zeta)}$. Hence
    \begin{align}
       |\sigma(x)| &\leq \sum_{z\in f^{-1}(w)} \frac{|\rho(z)|}{|z|^2\cdot |f'(z)|^2} \leq 
         \frac{\tilde{C}}{|w|^2}\cdot \sum_{m\in\Z} |\phi'({\zeta_0} + 2\pi i m)|^2 \label{eqn:doublepolefinalestimate}
   \\ 
       &\leq 
         \frac{8\tilde{C}(\re {\zeta_0})^2}{|w|^2}\cdot \sum_{m\in\Z} |\phi'( 1 + (2\pi m + \im \zeta)i)|^2 \notag
    \\
        &=
           \frac{8\tilde{C} (\log|w|)^2}{|w|^2}\cdot \sum_{z\in f^{-1}(\tilde{w})} 
                \frac{e^2}{|z|^2\cdot|f'(z)|^2} \notag
     \\ 
       &\leq 
         \const \cdot \frac{(\log|w|)^2}{|w|^2} 
            \max_{|\omega| = e} \sum_{z\in f^{-1}(\omega)} 
                \frac{1}{|z|^2\cdot|f'(z)|^2}. \notag
    \end{align}
    The maximum on the right-hand side 
      is finite by~\eqref{eqn:areatracts}. So indeed
      \[ |\sigma(x)| \leq \const\cdot \frac{(\log |w|)^2}{|w|^2} = o(1/|w|)\qquad\text{as $w\to\infty$}. \qedhere \]
 \end{proof}

We can deduce the following global statement.

\begin{cor}[Pushforwards of QD under class $\classS$ maps with the area property]\label{cor:doublepole}
  Let $f\in\classS$ have the area property, and let $q$ be a meromorphic quadratic differential on $\Ch$, with at most double poles.
   Then $f_*q$ is also a meromorphic quadratic differential on $\Ch$ with at most double poles.

 More precisely, $f_*q$ 
  \begin{itemize}
    \item has at most double poles at $\infty$, asymptotic values of $f$, and at the images of double poles of $q$;
    \item has at most simple poles at non-asymptotic critical values of $f$ and at the images of simple poles of $q$;
    \item is holomorphic elsewhere.
  \end{itemize}
\end{cor}
\begin{proof}
  Let us again write $\rho(z) = \rho(z) \frac{\dif z^2}{z^2}$, where $\rho$ is meromorphic on $\Ch$ and satisfies
     $\rho(z)= O(1)$ $w\to\infty$. Then 
    \[ f_* q = \left(\sum_{z\in f^{-1}(w)}\frac{\rho(z)}{f'(z)^2 z^2}\right) \dif w^2 =: \sigma(w)\dif w^2 \]  
   is defined and holomorphic by the area property, except possibly at singular values of $f$ (including $\infty$) and the images of
   poles of {$q$}. Let $E$ denote this finite exceptional set; we must investigate the behaviour of $\sigma$ near a point $w_0\in E$. 

  By postcomposing with a M\"obius transformation, we may assume that $w_0=0$ and that $E\cap\D = \emptyset$. Let
   $\V$ denote the set of all connected components of $f^{-1}(\D)$. For every $V\in\V$, denote the pushforward of $q$ under the restriction $f|_V$ by 
     $\sigma_V(w)\dif w^2$. Then $\sigma$ is defined and holomorphic on $\D^*$.
   \begin{enumerate}[(a)]
    \item If $f:V\to\D^*$ is a universal covering, then $\sigma_V$ has at most a double pole at $0$ by Proposition \ref{prop:doublepole}
     (applied to $1/f$). 
    \item Otherwise, $f:V\setminus f^{-1}(0)\to \D^*$ is a finite-degree covering map, and it follows that $V$ contains exactly one element $z_0$ of
      $f^{-1}(0)$, and that $f:V\to \D$ is a branched covering map with no critical points except possibly at $z_0$. The local pushforward of 
     a meromorphic quadratic differential defined near a critical point is well-understood (and can be verified by a simple computation in local coordinates);
     see \cite[Formula~(4)]{fatoushishikura}.
     \begin{itemize}
       \item The local pushforward of a quadratic differential with at most double poles under a holomorphic map has at most a double pole; hence
           $\sigma_V$ has at most a double pole at $0$.
       \item If $q$ has at most a simple pole at $z_0$, then $\sigma_V$ has at most a simple pole at $0$. (This also follows by considering the mass of
         $q$, as 
          mentioned above.) 
       \item If $q$ is holomorphic at $z_0$ and $f:V\to\D$ is a conformal isomorphism, then clearly $f_*q$ is holomorphic at $0$. 
     \end{itemize}
   \end{enumerate}

 So we have seen that each $\sigma_V$ has at most a double pole at zero. Recall that 
      \[ \sigma(w) = \sum_{V\in\V} \sigma_V(w) \]
     on $\D^*$, where the sum converges locally uniformly. It follows (e.g.\ by comparing the Laurent series of $\sigma(w)$ with that of the
    partial sums) that $\sigma$ also has at most a double pole at zero. The claim about simple poles follows analogously.
\end{proof}

  It
   is interesting to consider when $f_* q$ acquires at most simple poles also at finite asymptotic values or at infinity. For example, 
   set $f(z)\defeq e^z$ and $q\defeq \dif z^2/z^2$. We have 
   \begin{equation} \label{eqn:exppushforward}
f_{*}q=\sum_{k=-\infty }^{\infty }\frac{\dif w^{2}}{w^{2}(\log w+2\pi ik)^{2}}=%
\frac{-\pi \dif w^{2}}{4w^{2}\sin ^{2}\frac{\log w}{2i}}=\frac{\pi \dif w^{2}}{%
w^{3}-2w^{2}+w}
\end{equation}
 (where we used the infinite partial fraction expansion of $\frac{1}{\sin^2}$.) 
   Hence $f_*q$ indeed has only a \emph{simple} pole at $\infty$.   This phenomenon 
  means that the pushforward results in a massive cancellation of mass near infinity. 

  More generally, consider the following 
   strengthening of the area property: 
  \begin{equation}\label{eqn:strongarea}
      \sum_{z\in f^{-1}(w)\setminus\D} \frac{1}{|z|^{t} |f'(z)|^{t}}<\infty 
  \end{equation}
  for some $t=2-\eps<2$. A simple estimate shows that, for $f(z)=\exp(z)$, this property is satisfied for all $t>1$. Furthermore,
  if $f$ is a Poincar\'e function for a postcritically finite polynomial $p$, then~\eqref{eqn:strongarea} holds for some $t<2$. Indeed,
  as in inequality~\eqref{eqn:poincareseries} in the proof of Theorem \ref{thm:arealinearizers}, we see that the corresponding series is bounded by 
  the Poincar\'e series for $p$ with exponent $t$. This series converges when $t>\delta(p)$, where $\delta(p)$ is the critical
  exponent for the Poincar\'e series. For postcritically finite maps, it is known that $\delta(p)<2$ coincides with the Hausdorff dimension of the Julia set
   (compare \cite{przytyckiconical}).   

 \begin{thm}[Simple poles at asymptotic values]\label{thm:simplepole}
  Let $f\in\classS$ satisfy \eqref{eqn:strongarea} for some $t<2$, and let 
    $q$ be a meromorphic quadratic differential on the Riemann sphere with at most double poles.
   Then $f_*q$ is a meromorphic quadratic differential on the Riemann 
   sphere with at most double poles at the images of the finite double poles of $q$, and 
   at most simple poles elsewhere.
\end{thm}
The theorem follows by replacing Proposition \ref{prop:doublepole} in the proof of Corollary \ref{cor:doublepole} with the following
  observation.
  
\begin{prop}[Quadratic differentials and logarithmic singularities, II]\label{prop:simplepole}
  Suppose that $f$ and $q$ are as in Proposition \ref{prop:doublepole}, but that $f$ additionally satisfies~\eqref{eqn:strongarea}
   for some $t<2$. Then $f_*q$ has at most a simple pole at $0$. 
\end{prop}
\begin{proof}
 By a well-known estimate of Eremenko and Lyubich \cite[Lemma~1]{alexmisha}, we have
    \[ |z\cdot f'(z)| \geq \const\cdot |f(z)|\cdot \bigl| \log|f(z)| \bigr|\]
   for such a universal covering {when $|f(z)|$ is sufficiently large}.
   Hence,
    \[ \sum_{z\in f^{-1}(w)\setminus\D} \frac{1}{|z|^{2} |f'(z)|^{2}} 
          \leq \frac{1}{|w|^{\eps}} \cdot 
\sum_{z\in f^{-1}(w)\setminus\D} \frac{1}{|z|^{2-\eps} |f'(z)|^{2-\eps}}, \]
    provided that $\eps$ is chosen such that the sum on the left converges. (This is possible
     by~\eqref{eqn:strongarea}.) 
   In particular, if we replace the exponent $2$ by the exponent $2-\eps$ in the estimate~\eqref{eqn:doublepolefinalestimate} in the proof of Proposition~\ref{prop:doublepole}, we obtain 
    \[ |\sigma(x)| \leq  O((\log |w|)^2/|w|^{2+\eps}) = o(1/|w|^2), \] 
   showing that $f_* q$ has at most a simple pole.  
\end{proof}

\begin{remark}[Remark 1] Let $f$ be 
    the Poincar\'e function
    of a postcritically finite polynomial (or rational map) $p$ at a 
    repelling periodic point. Let
    $q=\dif z^2/z^2$; then 
    $q$ is invariant under the map $\lambda$ (multiplication by the multiplier of the
    repelling cycle). Using the functional relation~\eqref{eqn:functionalequation}, we see that
       \[ p_* f_* q = f_* \lambda_* q = f_* q, \]
     so $f_* q$ is pushforward invariant under $p$. (Note that
     the pushforward considered in~\eqref{eqn:exppushforward} is precisely of this type, using
     $f(z)=\exp(z)$ and $p(z)=z^2$.)

    We observed in Theorem \ref{thm:simplepole} that $f_* q$ has at most a 
    {\emph{simple}} pole
    at infinity (and at most double poles elsewhere). 
   The set of non-trivial
    quadratic differentials that are pushforward invariant under
    a postcritically finite polynomial or rational map is an intriguing
    object. {Unless the rational map in question is a 
    Latt\`es example, pushing forward a quadratic differential with at most simple poles
    eventually results in some cancellation of mass. (This is infinitesimal
    Thurston contraction \cite[Lemma~1 on p.~272]{douadyhubbardthurston}; compare also \cite[Lemma~3]{fatoushishikura}.) Hence a push-forward
    invariant quadratic differential}
    must have at least a double pole somewhere. Some quadratic differentials
    of this type can be obtained from algebraic expressions between
    multipliers of repelling cycles; {the above construction yields another
    set of examples.} It seems interesting
    to study their properties.
\end{remark}

\begin{remark}[Remark 2]
  Suppose that $f\in\B$ satisfies the area property (or its stronger variation~\eqref{eqn:strongarea}). 
    If $q$
   is a meromorphic quadratic differential with at most double pole at $\infty$, then 
   Propositions \ref{prop:doublepole} and \ref{prop:simplepole} still imply that 
    $f_*q$ is defined for sufficiently large values of $w$ and has at most a double (resp.\ simple pole) at $\infty$. However, there is no reason
    to expect the pushforward to be globally meromorphic. 
\end{remark}

\section{Further comments and questions}\label{sec:remarks}
 \subsection*{Meromorphic functions} We have stated our theorems, for the most part, for
   entire rather than meromorphic functions, as it is known classically that invariance of
   order can fail in the latter case. However, essentially all our results also apply, with the
   same proof, to
   meromorphic functions. (We remark that the area property was defined for
   entire functions, but it extends
   verbatim to the meromorphic case.) 
   In particular, Corollary \ref{cor:doublepole} and Theorem \ref{thm:simplepole} extend to meromorphic functions with finite
   singular sets, using exactly the same proof.

  An exception is given by Theorem \ref{thm:arealinearizers}, concerning Poincar\'e functions of polynomials.
 In the proof of this theorem, we used the fact that $p$ has nonempty Fatou set,
   but did not otherwise rely on the fact that $p$ is a polynomial. Hence Theorem
   \ref{thm:arealinearizers} holds 
   also for a Poincar\'e function $f$ of a transcendental entire function
    $h$ or of a rational function $h$, provided that the Fatou set of $h$ is non-empty. However,
    in the case where $F(h)=\emptyset$, it is possible for the area property to fail. Indeed, 
   consider the case where $h$ is a \emph{Latt\`es map}; that is, a (postcritically finite)
    rational function obtained from a linear toral endomorphism
   via projection to the Riemann sphere. By definition, the linearizer $f$ is an elliptic
   function, namely the projection from the torus in question to the sphere. Any
   doubly-periodic set of positive area has infinite cylindrical area, since 
     $\sum_{\lambda \in \Lambda\setminus\{0\}} 1/|\lambda|^2 = \infty$
     for any lattice $\Lambda\subset \C$. Hence $f$ does not have the 
   area property. 

 The general answer turns out to depend on the measurable dynamics of the function $h$.
   {Indeed, it is well-known that, for a Latt\`es map, almost every orbit is dense in the Riemann
    sphere. As the following result, which extends Theorem \ref{thm:arealinearizers}, shows, this is what causes
    the failure of the area property.}

\begin{thm}[Poincar\'e functions of rational or transcendental entire functions]
    Suppose that the transcendental meromorphic function $f$ is a Poincar\'e function of some
    entire or rational function 
     $h$. 

If $F(h)\neq \emptyset$, then $f$ has the area property {if and only if $h$ has no Siegel discs}. 
 Otherwise, 
    $f$ has the area property if and only if 
     $\dist(h^n(z),\P(h))\to 0$ for almost all $z\in\C$.
\end{thm}
\begin{proof}[Sketch of proof]
 {By the above remarks, it suffices to consider the case where $F(h)$ is empty.}
  As in the 
     proof of Theorem~\ref{thm:arealinearizers}, and as remarked after~\eqref{eqn:strongarea},
    the area property is equivalent to the
    question over whether the Poincar\'e series of $h$ converges for the exponent $2$ at 
    $w\notin \P(h)$. The fact that this is the case
    if and only if $\dist(h^n(z),\P(h))\to 0$ is surely known, at least for rational functions; 
   for completeness, we sketch a proof below. 

  If $\P(h)=\C$, then there is nothing to prove (since the area property only makes statements about sets \emph{disjoint} from the singular set of 
   $f$, it holds trivially). Otherwise, let $w\in\C\setminus \P(h)$. 
    By 
     the  existence of \emph{nice sets}
     (proved by Rivera in the rational case and Dobbs \cite{dobbsnice} in the general case),   
  there is a small simply-connected open set $U\subset\C\setminus \P(h)$ around $w$, 
    such that $h^n(\partial U)\cap \partial U = \emptyset$ for all $n\geq 1$. 
  The pullbacks of $U$ along first returns to $U$ form 
    an infinite conformal iterated function system (IFS). It is easy to see that the Poincar\'e series for $f$ corresponding to exponent $2$ will
    converge if and only if the sum, over all levels $n$, of the total area of the {sets of level $n$} in this IFS converges. 
   {Here by a ``set of level $n$'', we mean the result of applying a composition of $n$ of the contractions defining the 
   IFS. In other words, the sets of level $n$ are precisely the domains of the $n$-th return map to $U$.}

  If $\dist(h^n(z),\P(h))\to 0$ for almost all $z\in\C$, then there is a positive measure set of 
    points in $U$ that never return to $U$ under iteration. 	It follows that the areas in question decrease geometrically, and hence
    the Poincar\'e series converges. On the other hand, if there is a positive measure set of points with $\limsup \dist(h^n(z),\P(h))\not\to 0$, then
    by a well-known argument almost every orbit is dense (see e.g.\ \cite[Theorem~3.3]{linefields}). In particular, the level $n$ sets
    of the IFS have full area in $U$, and hence the Poincar\'e series diverges.
\end{proof}

\subsection*{The area property and measurable dynamics}
  In several places, our work suggests close connections between the area property
   and measurable dynamics. One such connection concerns the case of Poincar\'e functions, where we have
   seen that the measurable dynamics of the original map (here: the Poincar\'e series), are reflected in the value distribution (here: the area property and its 
   generalizations)
   for the linearizer. We remark that Eremenko and Sodin \cite{eremenkosodin} used this type of connection to give a new proof of the 
   existence of measures of maximal entropy for rational functions.

  Perhaps more interestingly, it appears that there are connections between conditions such as the area property and its
   stronger variant~\eqref{eqn:strongarea} and the measurable dynamics of the transcendental function itself. Indeed, we already saw that 
   the same construction that leads to a hyperbolic entire function with full hyperbolic dimension also yields a function for which the area property
   fails. Furthermore, such connections are suggested by work of Urba\'nski with several collaborators (see e.g.\ \cite{urbanskizdunik1,mayerurbanski}) 
   on the existence of conformal and invariant measures, and Hausdorff dimension of radial Julia sets, for finite-order entire and meromorphic functions.
   
 In particular, in \cite{mayerurbanski}, a class of hyperbolic meromorphic functions of finite order
   is treated that satisfy a strong regularity of growth property, known
   there as the \emph{balanced} condition. In the
   case of a finite-order entire function $f\in\B$, the function is balanced if and only if
   \begin{equation}\label{eqn:balanced} |f'(z)| \asymp (1+|z|)^{\rho(f)-1}\cdot (1+|f(z)|) \end{equation}
   for all $z\in J(f)$ \cite[Lemma~3.1]{mayerurbanski}. 
   It follows from this condition that~\eqref{eqn:strongarea} holds for all $t>1$, and hence all of these functions have the area
   property and the conclusions of Theorem \ref{thm:simplepole}. In particular, this implies that Poincar\'e functions of postcritically finite
   polynomials with non-smooth Julia sets are \emph{not} balanced in the sense of Mayer and Urba\'nski. Indeed, the critical exponent $c$ of the Poincar\'e series of such a polynomial coincides 
 with the Hausdorff dimension of its Julia set, which is larger than one. It follows that
   the Poincar\'e function does not satisfy~\eqref{eqn:strongarea} for $1<t<c$. 

 It is plausible that the minimal exponent in~\eqref{eqn:strongarea} is
   connected to the concept of \emph{eventual hyperbolic dimension}, which is defined in analogy with \cite[Section 5]{escapingdim} as follows:
 \[ \edim_{\hyp}(f) \defeq \lim_{R\to\infty} \sup\{\dim K\colon  K \text{ hyperbolic}, \min_{z\in K}|z|\geq R\}. \]
  (Recall that a set is \emph{hyperbolic} for $f$ if it is compact, invariant and expanding.)

\subsection*{Classes of functions with the area property}
  Given the many interesting connections between the area property and interesting applications in complex dynamics
   and function theory, it makes sense to identify classes of entire transcendental functions $f\in\B$ having the area property.
   As mentioned above, the balanced condition of Mayer and Urba\'nski \cite{mayerurbanski} gives rise to such a class of functions, but
   it is rather restrictive; in particular, it does not include Poincar\'e functions, which have been our primary source of non-trivial examples.
   We believe that there should be a natural \emph{geometric} condition on the tracts of an entire transcendental
   function $f\in\B$ that covers all balanced functions in this class (in the sense of~\eqref{eqn:balanced}) as well as 
   all Poincar\'e functions of postcritical polynomials. This condition should
    ensure that the area property holds and that quadratic differentials  with at most a double pole push forward to at most
   a simple pole at infinity. Furthermore, hyperbolic functions in the corresponding quasiconformal equivalence class should have hyperbolic
   dimension strictly less than two. As this
   question takes us beyond the scope of this note, it will be left to a subsequent paper.

\bibliographystyle{amsalpha}
\bibliography{../../Biblio/biblio}

\end{document}